\documentclass[a4paper,dvips,11pt]{article}
\usepackage[latin1]{inputenc}
\usepackage{amsmath,amsthm,amsfonts,amssymb}
\usepackage{enumerate}
\usepackage{graphicx,psfrag}
\usepackage[ps2pdf,colorlinks]{hyperref}
\usepackage{pst-node,pstricks-add,pst-text,pst-3d,pst-tree}

\numberwithin{equation}{section}

\theoremstyle{plain}
\newtheorem{thm}{Theorem}[section]

\newtheorem{lem}[thm]{Lemma}
\newtheorem*{thm*}{Theorem}
\newtheorem*{lem*}{Lemma}
\newtheorem*{prop*}{Proposition}

\theoremstyle{remark}

\newtheorem{remark}[thm]{Remark}

\newtheorem*{remark*}{Remark}

\newcommand{\mailto}[1]{\href{mailto:#1}{\nolinkurl{#1}}}

\newcommand{\be}{\begin{eqnarray}}
\newcommand{\ee}{\end{eqnarray}}
\newcommand{\beq}{\begin{equation}}
\newcommand{\eeq}{\end{equation}}
\newcommand{\beqn}{\begin{equation*}}
\newcommand{\eeqn}{\end{equation*}}

\newcommand{\R}{\mathbb{R}}
\newcommand{\Z}{\mathbb{Z}}

\newcommand{\defas}{\mathrel{\raise.095ex\hbox{$:$}\mkern-4.2mu=}}
\newcommand{\defasr}{\mathrel{=\mkern-4.2mu\raise.095ex\hbox{$:$}}}

\DeclareMathAlphabet{\mathfat}{U}{bbold}{m}{n}

\newcommand\cF{{\mathcal F}}

\newcommand\cU{{\mathcal U}}

\newcommand{\pre}{\mathbb{P}}
\newcommand{\Ee}{\mathbb{E}}

\newcommand{\prq}{P_\omega}
\newcommand{\Eq}{E_\omega}
\newcommand{\Varq}{\operatorname{Var}_\omega}

\newcommand{\E}{E}
\newcommand{\pr}{P}

\newcommand{\deq}{\stackrel{d}{=}}

\newcommand{\reff}[1]{(\ref{#1})}

\newcommand{\floor}[1]{[ #1 ]}
\newcommand{\ceil}[1]{\lceil #1 \rceil}

\newcommand{\ainv}{a^\leftarrow}

\begin{document}

\title{A local limit theorem for a transient chaotic walk\\ in a frozen environment}

\author{
 Lasse Leskelä\thanks{
 Postal address:
 Department of Mathematics and Systems Analysis, Aalto University, PO Box 11100, 00076 Aalto,
 Finland.
 URL: \url{http://www.iki.fi/lsl/} \quad
 Email: \protect\mailto{lasse.leskela@iki.fi}}
 \and
 Mikko Stenlund\thanks{
 Postal address:
 Courant Institute of Mathematical Sciences, New York, NY 10012, USA.
 URL: \url{http://www.math.helsinki.fi/mathphys/mikko.html} \quad
 Email: \protect\mailto{mikko@cims.nyu.edu}}
}
\date{\today}

\maketitle

\begin{abstract}
This paper studies particle propagation in a one-dimensional inhomogeneous medium where the laws
of motion are generated by chaotic and deterministic local maps. Assuming that the particle's
initial location is random and uniformly distributed, this dynamical system can be reduced to a
random walk in a one-dimensional inhomogeneous environment with a forbidden direction. Our main
result is a local limit theorem which explains in detail why, in the long run, the random walk's
probability mass function does not converge to a Gaussian density, although the corresponding
limiting distribution over a coarser diffusive space scale is Gaussian.
\end{abstract}


\vspace{1ex}


\vspace{1ex}

\section{Introduction}
\label{sec:Introduction}

\subsection{A chaotic dynamical system}

This paper studies a particle moving in a continuous inhomogeneous medium which is composed of a
linear chain of cells modeled by the unit intervals $[k,k+1)$ of the positive real line. Each
interval $[k,k+1)$ is assigned a label $\omega_k$ and a map $U_{\omega_k}$ which determines the
dynamics of the particle as long as the particle remains in the interval. The sequence of labels
$\omega = (\omega_k)_{k\in\Z_+}$, called an \emph{environment}, is assumed to be either nonrandom,
or a realization of a random sequence that is frozen during the particle's lifetime.

We are interested in the case in which the local dynamical rules $U_{\omega_k}$ are chaotic in the
sense that the distance between two initially nearby particles grows at an exponential rate. More
concretely, we shall focus on a model where a particle located at $x_n \in [k,k+1)$ at time $n$
jumps to $x_{n+1} = k + U_{\omega_k}(x_n - k)$. Here $\omega_k \in (0,1)$ and $U_{\omega_k}$
is the piecewise affine map from $[0,1)$ onto $[0,2)$ such that $U_{\omega_k}[0,1-\omega_k) =
[0,1)$ and $U_{\omega_k}[1-\omega_k,1) = [1,2)$. The dynamical system generated by the local rules
is compactly expressed by $x_{n+1} = \cU_{\omega}(x_n)$, where the global map $\cU_\omega$ on the
positive real line is defined by
\begin{equation}
 \label{eq:GlobalMap}
 \cU_\omega(x) = \floor{x} + U_{\omega_{[x]}}(x - \floor{x}),
\end{equation}
and $\floor{x}$ denotes the integral part of $x$, see Figure~\ref{fig:themap}.

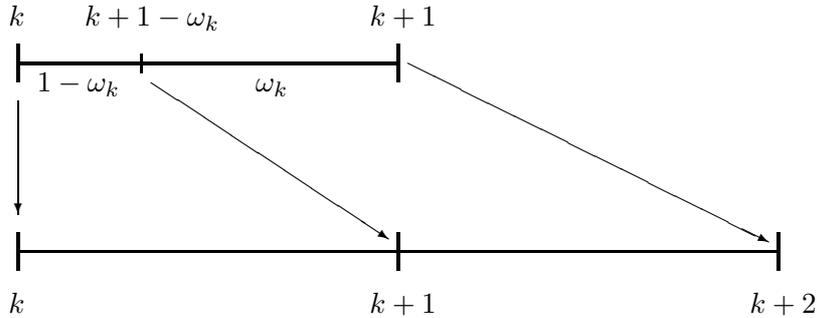
\begin{figure}[!ht]
\begin{center}
\setlength{\unitlength}{1.25cm}
\begin{picture}(8.5,4)
\thicklines
\put(0,3){\line(1,0){4}}
\put(0,2.8){\line(0,1){0.4}}
\put(4,2.8){\line(0,1){0.4}}
\put(1.3,2.9){\line(0,1){0.2}}
\put(-0.1,3.4){$k$}
\put(3.70,3.4){$k+1$}
\put(0.70,3.4){$k+1-\omega_k$}
\thinlines
\put(0,2.6){\vector(0,-1){1.2}}
\put(4.1,3){\vector(2,-1){3.8}}
\put(1.4,2.8){\vector(3,-2){2.5}}
\thicklines
\put(0,1){\line(1,0){8}}
\put(0,0.8){\line(0,1){0.4}}
\put(4,0.8){\line(0,1){0.4}}
\put(8,0.8){\line(0,1){0.4}}
\put(1.3,2.9){\line(0,1){0.2}}
\put(-0.1,0.35){$k$}
\put(3.70,0.35){$k+1$}
\put(7.70,0.35){$k+2$}
\put(0.2,2.7){$1-\omega_k$}
\put(2.5,2.7){$\omega_k$}
\end{picture}
\caption{The map $\cU_{\omega}$ acts piecewise affinely on the interval $[k,k+1)$.}
\label{fig:themap}
\end{center}
\end{figure}

The above model belongs to the realm of extended dynamical systems, a somewhat vaguely defined yet
highly active field of research (e.g.\ Chazottes and Fernandez~\cite{ChazottesFernandez}).
Telltale characteristics of such systems are a noncompact or high-dimensional phase space and the
lack of relevant finite invariant measures. Our principal motivation is to study the impact of
environment inhomogeneities on the long-term behavior of extended dynamical systems. Concrete
models include neural oscillator networks (Lin, Shea-Brown, and Young \cite{LinShea-BrownYoung})
and the Lorentz gas with randomly placed scatterers (Chernov and Dolgopyat
\cite{ChernovDolgopyat_ICM}; Cristadoro, Lenci, and Seri \cite{CristadoroLenciSeri}), to name a
few. In this paper, we shall restrict the analysis to the affine dynamical model
in~\reff{eq:GlobalMap}, to keep the presentation simple and clear.

\subsection{Random initial data}

Because the local maps $U_{\omega_k}$ are chaotic, predicting the particle's future location with
any useful accuracy over any reasonably long time horizon would require a precise knowledge of its
initial position --- a sheer impossibility in practice. Therefore, it is natural to take the
statistical point of view and study the stochastic process defined by
\begin{equation}
 \label{eq:DeterministicWalk}
 \begin{aligned}
  x_0     &\deq {\rm Uniform}[0,1), \\
  x_{n+1} &= \cU_\omega(x_n).
 \end{aligned}
\end{equation}

To analyze the time evolution of the above process, we must impose some regularity conditions on
the environment. In particular, those conditions guarantee ballistic motion, and one might guess that the distribution approaches Gaussian in the long run. To test this
hypothesis, we have plotted in Figure~\ref{fig:surprise} numerically computed histograms of $x_n$
at time $n=2^{13}$ in two frozen environments, using the intervals $[k,k+1)$ as bins. Rather
surprisingly, the histograms do not appear Gaussian. A similar phenomenon was recently observed by
Simula and Stenlund \cite{SimulaStenlund,SimulaStenlund2}.

\begin{figure}
\begin{minipage}[t]{\linewidth}
\includegraphics[width=\linewidth]{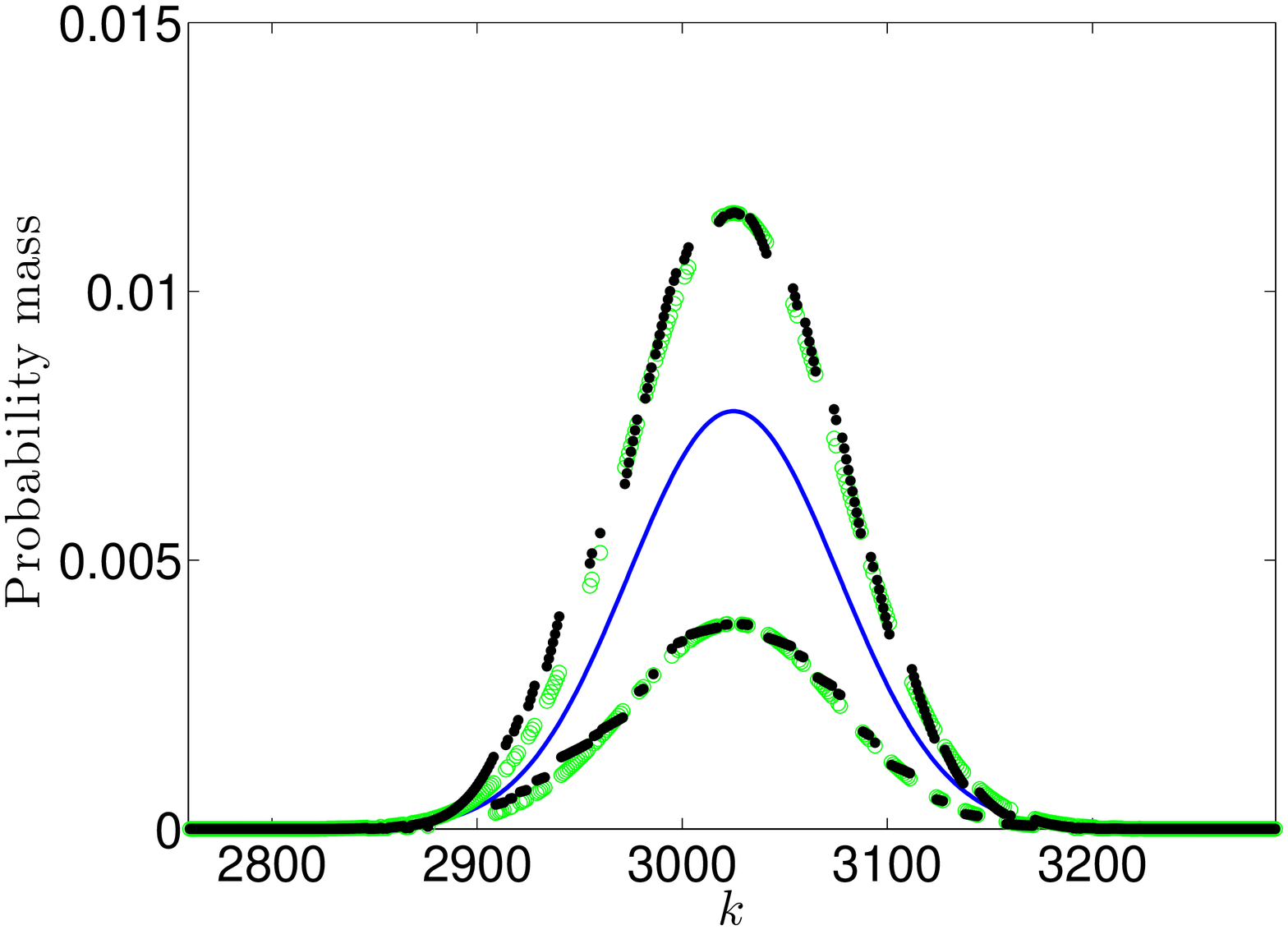}
\end{minipage}
\\
\begin{minipage}[t]{\linewidth}
\includegraphics[width=\linewidth]{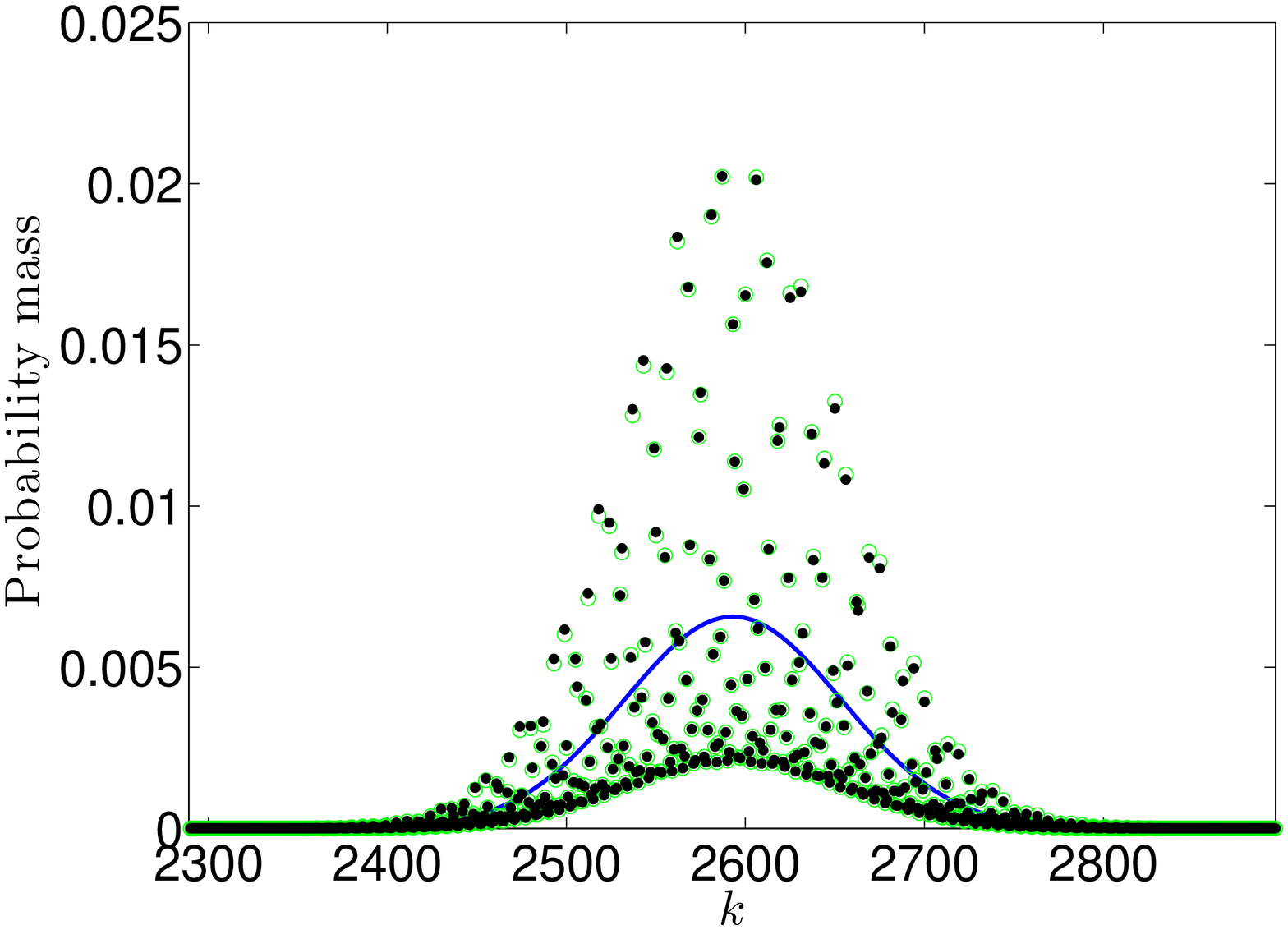}
\end{minipage}
\caption{Histograms of $x_n$ (black dots) at time $n=2^{13}$ in two frozen environments $\omega$.
Top: $\omega$ is a realization of a Markov chain with values in
$\left\{\frac14,\frac34\right\}$ and $\pre(\omega_{k+1} = \omega_{k}\,|\,\omega_k)=\frac45$.
Bottom: $\omega$ is nonrandom, with $\omega_k = \frac{11}{20} + \frac{9}{20} \sin k$.
The green circles are obtained by modulating the blue Gaussian density by the factor $\omega_k^{-1}/\mu$
appearing in Theorem~\ref{thm:LLT}.}
\label{fig:surprise}
\end{figure}

\subsection{Summary of main results}

The main contribution of the paper is to explain the emergence of the histograms in
Figure~\ref{fig:surprise}. This is accomplished by first reducing the continuum dynamical system
to a unidirectional random walk on the integers (Theorem~\ref{thm:MC}), and then deriving a local
limit law (Theorem~\ref{thm:LLT}) that completely explains the behavior observed in
Figure~\ref{fig:surprise}. As a byproduct, we also obtain a law of large numbers
(Theorem~\ref{thm:LLN}) and a central limit law (Theorem~\ref{thm:CLT}) for the walk. These limit
laws are valid for all frozen environments --- random or nonrandom --- which satisfy certain
statistical regularity properties. We also devote a separate section to the analysis of random
environments, where we show (Theorem~\ref{thm:Quenched}) that the three  aforementioned limit laws
are valid for almost all realizations of a stationary random environment under suitable moment and
mixing conditions. Because the random walk is
unidirectional (it never steps backwards), limit theorems for its hitting times are immediate
consequences of classical limit laws for independent random variables. Translating the limit
theorems of the hitting times into limit theorems of the walk location form the main task in
proving the results; see Section~\ref{sec:ProofLLTOutline} for a general outline of the proofs.

\subsection{Related work}

We shall discuss here only literature most closely related to transient one-dimensional random
walks in quenched random environments; for a broad picture of the theory of random walks in random
environments, see e.g.\ Bolthausen and Sznitman~\cite{BolthausenSznitman},
Sznitman~\cite{Sznitman}, and Zeitouni~\cite{Zeitouni}. Laws of large numbers and averaged central
limit theorems for random walks in random environments have been known already for a long time
(e.g.\ Solomon~\cite{Solomon_1975}; Kesten, Kozlov, and
Spitzer~\cite{Kesten_Kozlov_Spitzer_1975}), whereas the literature on quenched central limit
theorems is more recent. Buffet and Hannigan~\cite{Buffet_Hannigan_1991} proved a quenched central
limit theorem for a pure birth process in an independently scattered random environment under
moment conditions later relaxed by Horv{\'a}th and
Shao~\cite{Horvath_Shao_1994,Horvath_Shao_1995}; this model is a direct continuous-time analogue
of the random walk studied here. Quenched central limit theorems for more general one-dimensional
transient random walks were proved only very recently, independently by Goldsheid
\cite{Goldsheid2007} and Peterson \cite{PetersonThesis} (see also Alili \cite{Alili} for a result
concerning a special quasiperiodic environment). Rassoul-Agha and Sepp\"al\"ainen
\cite{RassoulAghaSeppalainen2006} obtained a similar result for multidimensional random walks with
a forbidden direction, which in the one-dimensional case corresponds to the unidirectional walk
analyzed in this paper. Dolgopyat, Keller, and Liverani \cite{DolgopyatKellerLiverani} have
obtained a quenched central limit theorem for environments changing in time and space.

Our approach differs from most earlier works in that we separately analyze the two degrees of
randomness involved in random walks in quenched random environments. In the first part, we extract
a set of statistical regularity properties for a given environment that are sufficient for proving
the limit laws, while treating the environment as nonrandom. In the second part, we derive
conditions for the probability distribution of the random environment that yield almost surely
regular realizations. A key result for the second part is a law of large numbers for the moving
averages of a stationary sequence (Lemma~\ref{lem:fast}), which is proved with the help of
Peligrad's extension \cite{Peligrad1985} of the Baum--Katz theorem~\cite{BaumKatz1965} (see
Bingham \cite{Bingham1986a} for a nice survey). The major advantage of this approach is a
clarified picture on how different sources of randomness affect the random walk's behavior in
quenched random environments.

Local limit theorems for random walks in homogeneous environments (e.g.\
Spitzer~\cite{Spitzer1976}) can be proved as simple consequences of Gnedenko's classical theorem
(e.g.\ \cite[Chapter 9]{GnedenkoKolmogorov1968} or \cite[Section 3.5]{Durrett2010}), because such
walks are just sums of independent random variables. A generalization for a periodic environment
was derived by Takenami \cite{Takenami2002}, and for a periodic graph recently by Kazami and
Uchiyama \cite{KazamiUchiyama}. In contrast, local limit theorems for random walks in aperiodic
inhomogeneous environments appear nonexistent. To the best of our knowledge,
Theorem~\ref{thm:LLT} and Theorem~\ref{thm:Quenched} are the first local limit laws concerning
transient random walks in aperiodic nonrandom or quenched random
environments. Although our analysis is restricted to a very special instance of a random walk, the model is still rich enough to capture
several interesting phenomena, such as the need for nonlinear centering and a non-Gaussian
modulating factor, and we believe that the results could serve as useful benchmarks when testing
hypotheses concerning more general random walks.

Regarding extended dynamical systems, we have found two earlier local limit theorems, both
corresponding to homogeneous environments. Sz{\'a}sz and Varj{\'u} \cite{SzaszVarju} have
considered Lorentz processes with periodic configurations of scatterers, while Bardet, Gou\"ezel,
and Keller~\cite{BardetGouezelKeller} study rather different type of systems: small (possibly
inhomogeneous) perturbations of weakly coupled, translation invariant, coupled map lattices; see
Nagaev~\cite{Nagaev1957} and Guivarc'h~\cite{Guivarc'h1984} for some of the original techniques.
Let us finally stress that for more classical, probability measure preserving, dynamical systems,
various types of limit theorems have been proved for many decades. Yet such systems, too, continue
to be studied vigorously, with important recent developments (e.g.\ Chazottes and Gou{\"e}zel
\cite{ChazottesGouezel}; Gou{\"e}zel~\cite{Gouezel2009,Gouezel2010}).

\subsection{Organization of the paper}
The rest of the paper is organized as follows. Section~\ref{sec:MainResults} presents the main
results. The proofs for nonrandom environments are given in Section~\ref{sec:ProofsNonrandom}, and
the proofs for quenched random environments in Section~\ref{sec:ProofsRandom}.
Section~\ref{sec:Conclusions} concludes the paper, and Appendix~\ref{sec:GeneralizedInverse}
contains basic facts on generalized inverses of increasing sequences.

\section{Main results}
\label{sec:MainResults}

\subsection{Representation as a random walk}
\label{sec:RandomWalk}

The distribution of the dynamical system~\reff{eq:DeterministicWalk} at any time instant can be
completely characterized in terms of the following simple unidirectional random walk
(discrete-time pure birth process) on the integers. Given an environment $\omega \in
(0,1)^{\Z_+}$, let $(X_n)_{n \in \Z_+}$ be a random walk in $\Z_+$ having the initial state $X_0 =
0$ and transitions
\begin{equation}
 \label{eq:RandomWalk}
 k \mapsto \left\{
 \begin{aligned}
   k,   &\quad \text{with probability} \ 1 - \omega_k, \\
   k+1, &\quad \text{with probability} \ \omega_k.
 \end{aligned}
 \right.
\end{equation}
We denote by $\prq$ the distribution of the walk in the path space $\Z_+^{\Z_+}$. Note that if the
environment $\omega$ is a realization of a random sequence, the process $(X_n)$ can be identified
as a random walk in a random environment, and the distribution $\prq$ is usually called the
\emph{quenched law} of the random walk. The expectation and variance with respect to $\prq$ are
denoted by $\Eq$ and $\Varq$, respectively. When presenting general facts in probability theory,
we write $\pr$ and $\E$ for the measure and expectation.

\begin{thm}
\label{thm:MC}
For any environment $\omega$, the value of the dynamical system~\reff{eq:DeterministicWalk} at any
time instant $n$ has the same distribution as $X_n + R$, where $(X_n)$ is the random walk in
$\Z_+$ defined by~\reff{eq:RandomWalk}, and $R$ is a uniformly distributed random variable in
$[0,1)$ independent of $X_n$.
\end{thm}

\begin{proof}
The proof follows by induction and Lemma~\ref{lem:T} below.
\end{proof}

\begin{lem}
\label{lem:T}
Let $X$ be a positive random variable such that (i) $\floor{X}$ and $\{X\} = X - \floor{X}$ are
independent, and (ii) $\{X\}$ is uniformly distributed in $[0,1)$. Then the same is true for
$\cU_\omega X$, and moreover,
\begin{equation}
 \label{eq:MarkovProperty}
 \pr( \floor{\cU_\omega X} = l \, | \, \floor{X} = k  ) = \left\{
 \begin{aligned}
   1 - \omega_k,     &\quad \text{if}  \ l = k, \\
   \omega_k, &\quad \text{if} \ l = k+1,
 \end{aligned}
 \right.
\end{equation}
whenever $\pr(\floor{X} = k) > 0$.
\end{lem}
\begin{proof}
Denote $X = K + R$, where $K$ is a positive random integer independent of $R$, and $R$ is
uniformly distributed in $[0,1)$. Assume first that $K=k$ for some nonrandom integer $k$. Then
$\cU_\omega X = k + U_{\omega_k} R$. Moreover, a simple calculation based on the definition of
$U_{\omega_k}$ shows that for all $r \in [0,1)$,
\begin{equation}
 \label{eq:T}
 \pr([\cU_\omega X] = l, \, \{\cU_\omega X\} \le r) = \left\{
 \begin{aligned}
   (1-\omega_k) r, &\quad \text{if} \ l=k, \\
   \omega_k r,     &\quad \text{if} \ l=k+1, \\
   0,              &\quad \text{else}.
 \end{aligned}
 \right.
\end{equation}
Hence $\cU_\omega X$ satisfies (i), (ii), and~\reff{eq:MarkovProperty} in the case where
$\floor{X}$ is nonrandom. The general case follows by conditioning on $X$.
\end{proof}

\subsection{Limit theorems for regular frozen environments}
\label{sec:NonrandomEnvironments}
In this section we shall analyze the random walk $(X_n)$ in a fixed, sufficiently regular
environment, which may either be nonrandom, or a realization of a random sequence. More precisely,
we shall in general assume that the environment $\omega \in (0,1)^{\Z_+}$ is such that
\begin{equation}
  \label{eq:SojournGrowth}
  \omega_k^{-1} = O(k^\lambda)
\end{equation}
for some $0 \le \lambda < 1/2$; and
\begin{align}
 \label{eq:SojournMean}
 k^{-1} \sum_{j=0}^{k-1} \omega_j^{-1}             &= \mu      + o(k^{-\lambda} (\log k)^{-1/2}), \\
 \label{eq:SojournVariance}
 k^{-1} \sum_{j=0}^{k-1} (1-\omega_j)\omega_j^{-2} &= \sigma^2 + o(k^{-\lambda} (\log k)^{-1/2}),
\end{align}
for some constants $\mu > 1$ and $\sigma^2 > 0$. Moreover, we assume that
\begin{equation}
 \label{eq:SojournThirdMoments}
  k^{-1} \sum_{j=0}^{k-1} \omega_j^{-3} = O(1),
\end{equation}
and that the environmental moving averages satisfy
\begin{equation}
 \label{eq:EnvMovingAverages}
 \max_{j:|j| \le u b(k)}
 \left| \sum_{\ell=k}^{k+j-1}(\omega_\ell^{-1}-\mu) \right| = o(k^{1/2-\lambda})
\end{equation}
for all $u>0$, where $b(k) = (k\log k)^{1/2}$. (In the special case with $\lambda=0$ it suffices
to use $b(k) = k^{1/2}$). Concrete examples of environments that satisfy the above regularity
conditions shall be given in Section~\ref{sec:RandomEnvironments}.

The quantity $\omega_k^{-1}$ represents the mean sojourn time of the particle in the interval
$[k,k+1)$ (see Section~\ref{sec:HittingTimes} for more details). Therefore, the constant $\mu$
appearing in~\reff{eq:SojournMean} may be interpreted as the inverse of the particle's traveling
speed. The following result confirms this intuition.

\begin{thm}[Law of large numbers]
\label{thm:LLN}
For any environment $\omega$ satisfying~\reff{eq:SojournGrowth} -- \reff{eq:SojournMean},
\begin{equation}
 \label{eq:LLN}
 \prq( n^{-1} X_n \to \mu^{-1}) = 1.
\end{equation}
\end{thm}

The main result of the paper is the following limit theorem for the random walk $(X_n)$ defined
by~\reff{eq:RandomWalk}, or alternatively (by virtue of Theorem~\ref{thm:MC}), for the dynamical
system defined by~\reff{eq:DeterministicWalk}.

\begin{thm}[Local limit theorem]\label{thm:LLT}
For any environment $\omega$ satisfying~\reff{eq:SojournGrowth} -- \reff{eq:EnvMovingAverages},
\begin{equation}
 \label{eq:LLT}
 \prq(X_n=k)
 \,=\,  \frac{\omega_k^{-1}}{\mu} \cdot \frac{1}{\sqrt{2\pi \tilde\sigma^2 n}} \, e^{-\frac{(k - k_n^\omega)^2}{2 \tilde\sigma^2 n}}
 \,+\, o(n^{-1/2}),
\end{equation}
uniformly with respect to $k \ge 0$, where $\tilde\sigma^2 = \sigma^2/\mu^3$, and the centering factors
$k_n^\omega$ are given by
\begin{equation}
 \label{eq:Centering}
 k_n^\omega = \min \left\{ k \ge 0: \sum_{j=0}^{k-1} \omega_j^{-1} \ge n \right\}.
\end{equation}
\end{thm}

Two features in Theorem~\ref{thm:LLT}, which distinguish it from classical limit theorems,
call for special attention. First, the centering factors $k_n^\omega$ depend on the environment,
and are in general nonlinear functions of $n$. Second, the modulating factor $\omega_k^{-1}/\mu$ in
\reff{eq:LLT} causes the asymptotic shape of the probability mass function of $X_n$ to be
non-Gaussian. This modulating factor explains the behavior observed in Figure~\ref{fig:surprise}.

In contrast, when looking at the probability distribution of the walk over a coarser diffusive
space scale, the non-Gaussian modulating factor in Theorem~\ref{thm:LLT} averages out
asymptotically, and we end up with a standard Gaussian limiting distribution.

\begin{thm}[Central limit theorem]
\label{thm:CLT}
For any environment $\omega$ satisfying~\reff{eq:SojournGrowth} -- \reff{eq:EnvMovingAverages} for
some constants $\mu>1$ and $\sigma^2>0$,
\begin{equation}
 \label{eq:CLT}
 \prq \! \left( \frac{X_n-k_n^\omega}{\tilde\sigma \sqrt{n}} \le x \right)
 \to \frac{1}{\sqrt{2\pi}} \int_{-\infty}^x  e^{-t^2/2} \, dt
\end{equation}
for all $x \in \R$, where $\tilde\sigma^2 = \sigma^2/\mu^3$, and $k_n^\omega$ are given
by~\reff{eq:Centering}.
\end{thm}

\begin{remark}\label{rem:LLT_cen1}
The centering factor $k_n^\omega$ may be identified as a generalized inverse
(Appendix~\ref{sec:GeneralizedInverse}) of the function $k \mapsto \Eq T_k$, where $T_k$ denotes
the hitting time of the walk into site $k$ (Section~\ref{sec:HittingTimes}). A quick inspection of
the proofs in Section~\ref{sec:ProofsNonrandom} shows that Theorem~\ref{thm:LLT} remains true if
$k_n^\omega$ is replaced by $\ell_n^\omega = k_n^\omega + o(n^{1/2-\lambda})$, where $\lambda$ is
as in~\reff{eq:SojournGrowth}; use the inequality $|e^{-x^2}-e^{-y^2}| \le |x-y|$ for $k\leq n$
and the proof of Lemma~\ref{lem:ExponentialBound} for $k>n$. Moreover, Theorem~\ref{thm:CLT}
remains true if $k_n^\omega$ is replaced by $\ell_n^\omega = k_n^\omega + o(n^{1/2})$.
\end{remark}


\subsection{Limit theorems for quenched random environments}
\label{sec:RandomEnvironments}

In this section we assume that the environment $\omega$ is a realization of a stationary random
sequence in $(0,1)^{\Z_+}$, and denote by $\pre$ its distribution on $(0,1)^{\Z_+}$. The
expectation with respect to $\pre$ is denoted by $\Ee$. We shall assume that
%
\begin{equation}
 \label{eq:EnvMoments}
 \Ee (\omega_0^{-1})^q < \infty
 \quad \text{for some $q>5$.}
\end{equation}
To guarantee that the environmental averages converge to their mean values rapidly
enough, we assume that
\begin{equation}
 \label{eq:EnvMixing}
 \sum_{k \ge 1} \phi^{1/2}(k) < \infty,
\end{equation}
where the mixing coefficients $\phi(k)$ are defined by
\begin{equation}
 \label{eq:MixingCoefficients}
 \phi(k) = \sup_{m} \sup_{A \in \cF_0^m, B \in \cF_{m+k}^\infty, \, \pre(A)>0} |\pre(B | A) - \pre(B)|,
\end{equation}
and where $\cF_0^m = \sigma(\omega_j, j \le m)$ and $\cF_m^\infty = \sigma(\omega_j, j \ge m)$
(e.g.\ Bradley \cite{Bradley2005}).

The following result summarizes three limit theorems for the quenched random walk in a stationary
strongly mixing random environment.

\begin{thm}
\label{thm:Quenched}
The law of large numbers~\reff{eq:LLN}, the local limit law~\reff{eq:LLT}, and the central limit
law~\reff{eq:CLT} are valid with $\mu = \Ee \omega_0^{-1}$ and $\sigma^2 = \Ee
(1-\omega_0)\omega_0^{-2}$ for almost every realization of a stationary random environment
satisfying~\reff{eq:EnvMoments} and~\reff{eq:EnvMixing}.
\end{thm}

Especially, the limit laws summarized by Theorem~\ref{thm:Quenched} hold in the following cases:
\begin{itemize}
  \item Independently scattered stationary environments (environments where the site labels $\omega_k$ are
  independent and identically distributed).
  \item Uniformly ergodic environments as discussed in Goldsheid~\cite{Goldsheid2007}.
  \item Environments which are realizations of finite-state irreducible aperiodic stationary
      Markov chains, or more general Markov chains satisfying Doeblin's condition (e.g.\
      \cite{Bradley2005}).
\end{itemize}
Although in many applications it is natural to assume that the environment is stationary, the
limits of Theorem~\ref{thm:Quenched} remain valid under looser conditions, as is clear from the
results of Section~\ref{sec:NonrandomEnvironments}.

Alternative versions of the law of large numbers and the central limit law in
Theorem~\ref{thm:Quenched}, where the centering factors $k_n^\omega$ are replaced by $\Eq X_n$, can
be proved as consequences of Theorems~3.1 and~5.4 in Rassoul-Agha and Sepp\"al\"ainen
\cite{RassoulAghaSeppalainen2006}, if we additionally assume that
\[
 \pre\Bigl( \inf_k \omega_k \ge \delta\Bigr) = 1 \quad \text{for some $\delta > 0$}.
\]
This so-called nonnestling assumption is close in spirit to the uniform ellipticity of
nearest-neighbor random walks in random environments; in the context of our model it corresponds to
the special case $\lambda = 0$ in \reff{eq:SojournGrowth}. Although we believe that the local and
central limit laws in Theorems \ref{thm:LLT}, \ref{thm:CLT}, and \ref{thm:Quenched} remain
generally valid also for the alternative centering $n \mapsto \Eq X_n$, we prefer to use the
centering $n \mapsto k_n^\omega$ defined in~\reff{eq:Centering}, because these factors are easily
computed from the environment. Analogous central limit laws for nearest-neighbor walks were
recently independently found by Goldsheid \cite{Goldsheid2007} and Peterson \cite{PetersonThesis}.

If we were only interested in the law of large numbers~\reff{eq:LLN}, we could do with less
assumptions in Theorem~\ref{thm:Quenched}. For example, as our proof in
Section~\ref{sec:ProofsRandom} shows, the moment condition \reff{eq:EnvMoments} would only be
needed for $q>2$. The mixing assumption \reff{eq:EnvMixing} could be relaxed as well, see for
example Bingham \cite{Bingham1986a} for more details.

\section{Proofs for regular nonrandom environments}
\label{sec:ProofsNonrandom}

This section is devoted to proving Theorems \ref{thm:LLN}--\ref{thm:CLT} for the random walk
$(X_n)$ in an environment that satisfies the regularity assumptions \reff{eq:SojournGrowth} --
\reff{eq:EnvMovingAverages}. The environment $\omega$ shall be fixed once and for all during the
whole section --- here we do not care whether it is a realization of a random sequence or not.

The section is organized as follows. Section~\ref{sec:HittingTimes} describes some preliminaries
on the hitting times of the walk, and Section~\ref{sec:ProofLLN} gives the proof of the law of
large numbers. The proof of the local limit theorem is split into Sections
\ref{sec:ProofLLTOutline}--\ref{sec:ProofLLT}, and the proof of the central limit theorem is in
Section~\ref{sec:ProofCLT}.

\subsection{Hitting times of the walk}
\label{sec:HittingTimes}

We denote the hitting time of $(X_n)$ into site $k$ by $T_k = \min\{n \ge 0: X_n = k\}$, and the
sojourn time at site $k$ by $\tau_k = T_{k+1}-T_k$. Because the walk never moves backwards, the
equivalence
\begin{equation}
 \label{eq:LocationVSHitting}
 X_n = k
 \quad \text{if and only if} \quad
 T_k \le n < T_{k+1}
\end{equation}
is valid for all $k$ and $n$. Moreover, the sojourn times are independent, and $\tau_k$ has a
geometric distribution on $\{1,2,\dots\}$ with success probability $\omega_k$. Hence the mean and
the variance of $\tau_k$ are given by $\Eq \tau_k = \omega_k^{-1}$ and $\Varq(\tau_k) =
(1-\omega_k)\omega_k^{-2}$, respectively. The mean and the variance of $T_k$ are denoted by $\mu_k
= \Eq T_k$ and $\sigma^2_k = \Varq(T_k)$, so that
\begin{equation}
 \label{eq:MuSigma}
 \mu_k       = \sum_{j=0}^{k-1} \omega_j^{-1}
 \quad \text{and} \quad
 \sigma_k^2  = \sum_{j=0}^{k-1} (1-\omega_j)\omega_j^{-2}.
\end{equation}
The following result transforms the realization-by-realization relationship
\reff{eq:LocationVSHitting} into one concerning the probability mass functions.

\begin{lem}
\label{lem:Space2Time}
For any environment $\omega$ and any $k,n \ge 0$,
\[
 \prq(X_n = k) = \omega_k^{-1} \prq(T_{k+1} = n+1).
\]
\end{lem}
\begin{proof}
Note that $\prq(\tau_k > m) = \omega_k^{-1} \prq(\tau_k = m+1)$ for all $m \ge 0$. Because $T_k$
and $\tau_k$ are independent, we find by applying \reff{eq:LocationVSHitting} and conditioning on
$T_k$ that
\begin{align*}
 \prq(X_n = k)
 &= \prq(T_k \le n, \, T_k + \tau_k > n) \\
 &= \Eq 1_{\{T_k \le n\}} \prq(T_k + \tau_k > n \, | \, T_k) \\
 &= \omega_k^{-1} \Eq 1_{\{T_k \le n\}} \prq(T_k + \tau_k = n + 1 \, | \, T_k) \\
 &= \omega_k^{-1} \prq( T_k \le n, T_k + \tau_k = n + 1).
\end{align*}
This implies that claim, because $T_{k+1} = T_k + \tau_k$ and $\tau_k \ge 1$ almost surely.
\end{proof}

\subsection{Proof of Theorem~\ref{thm:LLN}}
\label{sec:ProofLLN}

By~\reff{eq:SojournGrowth} we see that $\Varq(\tau_k) \le c k^{2\lambda}$ for some constants $c
> 0$ and $\lambda < 1/2$, so that $\sum_{k=0}^\infty k^{-2} \Varq(\tau_k) < \infty$, which implies
as a consequence of Kolmogorov's variance criterion (e.g.\ \cite[Cor.
4.22]{Kallenberg2002}) that $k^{-1} \sum_{j=0}^{k-1} (\tau_j - \omega_j^{-1}) \to 0$ almost
surely. By~\reff{eq:SojournMean}, we conclude that the hitting times of the walk satisfy the
following law of large numbers:
\begin{equation}
 \label{eq:LLNHittingTimes}
 k^{-1} T_k \to \mu \quad \text{almost surely}.
\end{equation}
Further, by~\reff{eq:LocationVSHitting}, we see that $X_n = T^\leftarrow(n+1)-1$, where
$T^\leftarrow$ denotes the generalized inverse of the sequence $(T_k)$ defined in
Appendix~\ref{sec:GeneralizedInverse}. Lemma~\ref{lem:InverseMuLimit} together with
\reff{eq:LLNHittingTimes} now shows that $n^{-1} T^\leftarrow(n) \to \mu^{-1}$ almost surely, and
the proof of Theorem~\ref{thm:LLN} is complete. \qed

\subsection{Outline of proof for the local limit theorem}
\label{sec:ProofLLTOutline}

The starting point of the proof is Lemma~\ref{lem:Space2Time} in Section~\ref{sec:HittingTimes},
which reduces the problem into analyzing the probability mass function of the hitting times $T_k$.
Because $T_k$ is a sum of independent random variables, we may apply a classical local
limit theorem (e.g.\ Petrov~\cite{Petrov}) in
Section~\ref{sec:LLTHittingTimes} to conclude that $\prq(T_k = n) \approx f_k(n)$ for large values
of $k$, where
\begin{equation}
 \label{eq:f}
 f_k(n) = (2\pi \sigma_k^2)^{-1/2} \, \, e^{-\frac{(n-\mu_k)^2}{2\sigma_k^2}},
\end{equation}
and $\mu_k$ and $\sigma^2_k$ denote the mean and the variance of $T_k$ given by~\reff{eq:MuSigma}.

In the rest of the proof we need to transform the Gaussian density $f_k$ of the time variable
into a Gaussian density of the space variable. This will be accomplished in two steps. We show in
Section~\ref{sec:fg} that $f_k(n) \approx g_k(n)$, where
\begin{equation}
 \label{eq:g}
 g_k(n) = (2\pi n\sigma^2 / \mu)^{-1/2} \, \, e^{-\frac{(\mu_k-n)^2}{2 n\sigma^2 / \mu}},
\end{equation}
and further in Section~\ref{sec:hit_loc} that $g_k(n) \approx \mu^{-1} h_n(k)$, where
\begin{equation}
 \label{eq:h}
 h_n(k) = (2\pi n \sigma^2 / \mu^3)^{-1/2} \, \, e^{-\frac{(k-k_n)^2}{2n \sigma^2/\mu^3}}
\end{equation}
is the Gaussian density appearing in the statement of Theorem~\ref{thm:LLT} (we write $k_n$ in
place of $k_n^\omega$ for convenience).

The approximation $g_k(n) \approx \mu^{-1} h_n(k)$ is a subtle part in the argument, where we have
been guided by the following intuition. By approximating $\mu_{k_n} \approx n$, we obtain $\mu_k -
n \approx \mu_k - \mu_{k_n} = \sum_{\ell = k_n}^{k_n+j-1} \omega_\ell^{-1}$, where $j = k-k_n$.
Therefore, the difference of the square roots of the exponents in $\eqref{eq:g}$ and
$\eqref{eq:h}$ can be approximated by
\[
 \frac{(\mu_k-n)}{\sqrt{2 n\sigma^2 / \mu}} - \frac{\mu(k-k_n)}{\sqrt{2 n\sigma^2 / \mu}}
 \, \approx \,
 (2 n \sigma^2 / \mu)^{-1/2} \sum_{\ell=k_n}^{k_n+j-1} (\omega_\ell^{-1} - \mu) .
\]
Condition \eqref{eq:EnvMovingAverages} has been tailored to guarantee that the above difference is
small for values of $k$ such that $\mu_k \approx n$.

An essential technical complication is the fact that the modulating factor $\omega_k^{-1}$
in~\eqref{eq:LLT} (see also Lemma~\ref{lem:Space2Time}) is unbounded. To overcome this difficulty,
we have strived to obtain sharp estimates. Identifying the set of reasonable sufficient
assumptions \reff{eq:SojournGrowth} -- \reff{eq:EnvMovingAverages} has played a crucial role in
the proof.

\subsection{Local limit theorem for the hitting times}
\label{sec:LLTHittingTimes}

By applying a classical local limit theorem  for sums of independent random variables (Petrov
\cite[Theorem VII.5]{Petrov}; see also Davis and McDonald \cite{Davis1995}),
we obtain the following local limit theorem for the hitting times.

\begin{lem}
\label{lem:LLTHittingTimes}
For any environment $\omega$ satisfying \reff{eq:SojournVariance} and
\reff{eq:SojournThirdMoments},
\begin{equation}
 \label{eq:LLTHittingTimes}
 \sup_{n \ge 0} \left| \prq(T_k=n) - f_k(n) \right| = O(k^{-1}),
\end{equation}
where the functions $f_k$ are defined by~\reff{eq:f}.
\end{lem}
\begin{proof}
Observe first that by~\reff{eq:SojournVariance},
\[
 \liminf_{k \to \infty} k^{-1} \sum_{j < k} \Varq(\tau_j) = \sigma^2 > 0,
\]
and that
\[
 \limsup_{k \to \infty} k^{-1} \sum_{j < k} \Eq | \tau_j - \Eq \tau_j |^3 < \infty
\]
because of $\Eq | \tau_j - \Eq \tau_j |^3 \le 3 \omega_j^{-3}$ and~\reff{eq:SojournThirdMoments}.
Therefore, to apply \cite[Theorem VII.5]{Petrov}, it suffices to verify (note that $\prq(\tau_j =
1) \ge \prq(\tau_j = m)$ for all $j$ and $m$) that
\begin{equation}
 \label{eq:Petrov}
 \frac{1}{\log k} \sum_{j<k} \prq(\tau_j=1) \prq(\tau_j=2)
 \to \infty.
\end{equation}
By writing $\omega_j^{-2}(1-\omega_j) = x_j y_j$, where $x_j = \omega_j^{2/5} (1-\omega_j)^{1/5}$
and $y_j = \omega_j^{-12/5} (1-\omega_j)^{4/5}$, and applying H\"older's inequality with conjugate
exponents $5$ and $5/4$, we see that
\[
 \sum_{j<k} \omega_j^{-2}(1-\omega_j)
 \le \left( \sum_{j<k} \omega_j^2 (1-\omega_j) \right)^{1/5}
  \! \left( \sum_{j<k} \omega_j^{-3} (1-\omega_j) \right)^{4/5}.
\]
After dividing both sides above by $k$, and applying \reff{eq:SojournVariance} and
\reff{eq:SojournThirdMoments}, we see that
\[
 \liminf_{k \to \infty} k^{-1} \sum_{j<k} \omega_j^2 (1-\omega_j) > 0,
\]
which implies \reff{eq:Petrov}, because $\prq(\tau_j=1) \prq(\tau_j=2) = \omega_j^2 (1-\omega_j)$.
The claim now follows by applying \cite[Theorem VII.5]{Petrov} and recalling that $\sigma_k^{-1} =
O(k^{-1/2})$ by~\reff{eq:SojournVariance}.
\end{proof}

The following result is an analogue of Lemma \ref{lem:LLTHittingTimes}, where the time variable $n$
instead of the space variable $k$ tends to infinity.

\begin{lem}
\label{lem:LLTHittingTimes2}
For any $\lambda \in [0,1/2)$ and any environment $\omega$ satisfying \reff{eq:SojournMean},
\reff{eq:SojournVariance}, and \reff{eq:SojournThirdMoments}, the Gaussian densities $f_k$ defined
by~\reff{eq:f} satisfy
\[
 \sup_{k \ge 1} k^\lambda | \prq(T_k = n) - f_k(n) | = O(n^{\lambda-1}).
\]
\end{lem}
\begin{proof}
Fix an $\epsilon \in (0,\mu^{-1})$, and note that
\begin{equation}
 \label{eq:large_k}
 \sup_{k \ge \epsilon n} k^\lambda | \prq(T_k=n) - f_k(n) | = O(n^{\lambda-1})
\end{equation}
by Lemma~\ref{lem:LLTHittingTimes}. Therefore, we only need to analyze $\prq(T_k=n)$ and $f_k(n)$
for $1 \le k \le \epsilon n$. As a preliminary, note that, because $\epsilon < \mu^{-1}$ and
$\mu_n/n \to \mu$ by \reff{eq:SojournMean}, we may fix a constant $c_1 < 1$ and an integer $n_0$
such that $\mu_{\floor{\epsilon n}} \le c_1 n$ for all $n \ge n_0$.

Assume now that $k \le \epsilon n$ and $n \ge n_0$. Then $n - \mu_k \ge n - \mu_{[\epsilon n]} \ge
(1-c_1) n$, and moreover, $\sigma^2_k \le \sigma^2_+ \epsilon n$, where $\sigma^2_+ = \sup_{\ell
\ge 1} (\sigma^2_\ell/\ell)$ is finite by \reff{eq:SojournVariance}. Therefore,
\begin{equation}
 \label{eq:small_k1}
 (n-\mu_k)^2/\sigma_k^2 \ge c_2 n,
\end{equation}
where $c_2 = (1-c_1)^2/(\sigma^2_+ \epsilon) > 0$. A rough estimate together with Chebyshev's
inequality shows that
\[
 \prq(T_k = n)
 \le \prq(|T_k - \mu_k| \ge |n-\mu_k|)
 \le (n-\mu_k)^{-2} \sigma^2_k,
\]
so by \reff{eq:small_k1}, we conclude that
\begin{equation}
 \label{eq:small_k2}
 k^\lambda \prq(T_k = n)
 \le c_2^{-1} \epsilon^\lambda n^{\lambda-1}.
\end{equation}
By applying \reff{eq:small_k1} once more, we see that
\begin{equation}
 \label{eq:small_k3}
 k^\lambda f_k(n)
 \le (2 \pi \sigma^2_-)^{-1/2} k^{\lambda-1/2} e^{-c_2 n/2}
 \le (2 \pi \sigma^2_-)^{-1/2} e^{-c_2 n/2},
\end{equation}
where $\sigma^2_- = \inf_{\ell \ge 1} (\sigma_\ell^2/\ell)$ is strictly positive by
\reff{eq:SojournVariance}. The proof is now completed by combining the estimates \reff{eq:small_k2}
and \reff{eq:small_k3} with \reff{eq:large_k}.
\end{proof}

\subsection{Variance of the hitting times}
\label{sec:fg}

\begin{lem}
\label{lem:fg}
For any $\lambda \in [0,1/2)$ and any environment $\omega$ satisfying \reff{eq:SojournMean} and
\reff{eq:SojournVariance},
\begin{equation}
 \label{eq:fg}
 \max_{1 \le k \le n} k^\lambda | f_k(n) - g_k(n) | = o(n^{-1/2}),
\end{equation}
where the functions $f_k$ and $g_k$ are defined by~\reff{eq:f} and~\reff{eq:g}, respectively.
\end{lem}
\begin{proof}

The proof is split into two parts according to whether or not $|n-\mu_k| \le u (n \log n)^{1/2}$,
where $u$ is a large constant to be determined later.

(i) Assume that $k \le n$ is such that $|n-\mu_k| \le u (n \log n)^{1/2}$. Using the triangle
inequality and the inequality $|e^{-x^2}-e^{-y^2}| \le |x-y|$, we see that
\begin{align*}
 &(2\pi n)^{1/2} \left| f_k(n) - g_k(n) \right| \\
 & \quad \le \left| (n/\sigma_k^2)^{1/2} - (\mu / \sigma^2)^{1/2} \right| + (\mu/\sigma^2)^{1/2}
   \left| e^{-\frac{(n-\mu_k)^2}{2\sigma_k^2}} - e^{-\frac{(n-\mu_k)^2}{2 n\sigma^2 / \mu}} \right| \\
 & \quad \le \left( 1 + (2 \sigma^2/\mu)^{-1/2} \frac{|n-\mu_k|}{\sqrt{n}} \right)
   \left| (n/\sigma_k^2)^{1/2} - (\mu / \sigma^2)^{1/2} \right|.
\end{align*}
Consequently, assuming that $n$ is large enough so that $u (\log n)^{1/2} \ge 1$,
\[
 (2\pi n)^{1/2} \left| f_k(n) - g_k(n) \right|
 \le c_1 u (\log n)^{1/2} \left| (n/\sigma_k^2)^{1/2} - (\mu / \sigma^2)^{1/2} \right|
\]
where $c_1 = 1 + (2\sigma^2/\mu)^{-1/2}$.

Observe next that, assuming $n$ is large enough so that $n - u (n \log n)^{1/2} \ge n/2$,
\begin{equation}
 \label{eq:kLower}
 k = (\mu_k/k)^{-1} (n - (n-\mu_k)) \ge (2 \mu_+)^{-1} n,
\end{equation}
where $\mu_+ = \sup_{\ell \ge 1} (\mu_\ell/\ell)$. Further,
\[
 \left| (n/\sigma_k^2)^{1/2} - (\mu / \sigma^2)^{1/2} \right|
 = \frac{| n/\sigma^2_k - \mu/\sigma^2 |}{(n/\sigma^2_k)^{1/2} + (\mu/\sigma^2)^{1/2}}
 \le (\sigma^2/\mu)^{1/2} | n/\sigma^2_k - \mu/\sigma^2 |.
\]
Now using \reff{eq:kLower} we find that $\sigma_k^2 \ge c_2^{-1}n$ for $c_2 = 2\mu_+/\sigma^2_-$.
Therefore,
\[
 \left| n/\sigma_k^2 - \mu / \sigma^2 \right|
 = |(n-\mu_k)/\sigma^2_k + d(k)|
 \le c_2 u n^{-1/2} (\log n)^{1/2} + |d(k)|,
\]
where
\[
 d(k) = \frac{\mu_k/k}{\sigma_k^2/k} - \frac{\mu}{\sigma^2}.
\]
As a consequence,
\begin{equation}
 \label{eq:fg1}
 n^{1/2} k^\lambda \left| f_k(n) - g_k(n) \right|
 \le c_3 u n^\lambda (\log n)^{1/2} \left( u n^{-1/2} (\log n)^{1/2} + |d(k)| \right)
\end{equation}
for all large enough $n$, where $c_3 = c_1 \max(c_2,1)$. By combining \reff{eq:SojournMean} and
\reff{eq:SojournVariance}, we find that $|d(k)| = o(k^{-\lambda} (\log k)^{-1/2})$. Because $(2
\mu_+)^{-1} n \le k \le n$ by \reff{eq:kLower}, we conclude that the right side above tends to zero
as $n \to \infty$.

(ii) Assume now that $k \le n$ is such that $|n-\mu_k| > u (n \log n)^{1/2}$. Note that $\sigma^2_-
k \le \sigma^2_k \le \sigma^2_+ k$, where $\sigma^2_- = \inf_{\ell \ge 1} (\sigma_\ell^2/\ell)$ and
$\sigma^2_+ = \sup_{\ell \ge 1} (\sigma^2_\ell/\ell)$ are finite and strictly positive by
\reff{eq:SojournVariance}. Therefore, the exponent in the definition of $f_k$ is bounded by
\[
 \frac{(\mu_k - n)^2}{2 \sigma^2_k} \ge (2 \sigma^2_+)^{-1} u^2 \log n.
\]
Consequently,
\begin{equation}
 \label{eq:fg2}
 n^{1/2} k^\lambda f_k(n)
 \le c_4 k^{\lambda - 1/2} n^{1/2 - c_5 u^2}
 \le c_4 n^{1/2 - c_5 u^2},
\end{equation}
where $c_4 = (2 \pi \sigma^2_-)^{-1/2}$ and $c_5 = (2 \sigma^2_+)^{-1}$. For the function $g_k$ we
immediately see that
\begin{equation}
 \label{eq:fg3}
 n^{1/2} k^\lambda g_k(n) \le c_6 n^{1/2 - c_7 u^2},
\end{equation}
where $c_6 = (2\pi \sigma^2/\mu)^{-1/2}$ and $c_7 = (2 \sigma^2/\mu)^{-1}$. The proof is now
finished by choosing $u>0$ large enough so that $c_5 u^2 > 1/2$ and $c_7 u^2 > 1/2$, and combining
the estimates \reff{eq:fg2} and \reff{eq:fg3} with \reff{eq:fg1}.
\end{proof}

\subsection{From hitting times to walk locations}
\label{sec:hit_loc}

\begin{lem}
\label{lem:gh}
For any environment $\omega$ satisfying \reff{eq:SojournMean} and~\reff{eq:EnvMovingAverages} for
some $\lambda \in [0,1/2)$,
\[
 \max_{1 \le k \le n} k^\lambda | g_k(n) - \mu^{-1} h_n(k) | = o(n^{-1/2}),
\]
where the functions $g_k$ and $h_n$ are defined by~\reff{eq:f} and~\reff{eq:h}, respectively.
\end{lem}
\begin{proof}
We will show the claim by treating separately the cases $|k-k_n| \le u b(k_n)$ and $|k-k_n| > u
b(k_n)$, where $b(k) = (k \log k)^{1/2}$ is as in \reff{eq:EnvMovingAverages}, and $u>0$ is a large
constant to be determined later.

(i) Assume that $k \le n$ is such that $|k-k_n| \le u b(k_n)$. Using the inequality $|e^{-x^2} -
e^{-y^2}| \le |x-y|$, we see that
\[
 | g_k(n) - \mu^{-1} h_n(k) |
 \le c_1 n^{-1} | d(k,n) |,
\]
where $c_1 = \pi^{-1/2} (2 \sigma^2/\mu)^{-1}$, and $d(k,n) = (n-\mu_k) - \mu(k_n - k)$. Note that
\begin{equation}
 \label{eq:InverseDiff1}
 |d(k,n)|
 = \left| \sum_{\ell=k}^{k_n-1}(\omega_\ell^{-1}-\mu) + n - \mu_{k_n} \right|
 \le \left| \sum_{\ell=k}^{k_n-1}(\omega_\ell^{-1}-\mu) \right| + \omega^{-1}_{k_n-1},
\end{equation}
where the latter inequality is due to $\mu_{k_n} - \omega_{k_n-1}^{-1} < n \le \mu_{k_n}$. To
analyze the right side of \eqref{eq:InverseDiff1}, observe that, because $k_n \to \infty$, we see
using~\reff{eq:EnvMovingAverages} that
\[
 \max_{k:|k-k_n| \le u b(k_n)} \left| \sum_{\ell=k}^{k_n-1} (\omega_\ell^{-1}-\mu) \right|
 =
 \max_{j:|j| \le u b(k_n)} \left|  \sum_{\ell=k_n}^{k_n+j-1} (\omega_\ell^{-1}-\mu) \right|
 = o(k_n^{1/2 - \lambda}).
\]
The above limiting relation also shows (substitute $j=1$) that $\omega^{-1}_{k_n-1} =
o(k_n^{1/2-\lambda})$. Because $k_n/n \to \mu^{-1}$ (by~\eqref{eq:SojournMean}
and Lemma~\ref{lem:InverseMuLimit}), it follows that
\[
 \max_{k: |k-k_n| \le u b(k_n)} | d(k,n) | = o(n^{1/2-\lambda}),
\]
and therefore,
\begin{equation}
 \label{eq:fh2}
 n^{1/2} \max_{1 \le k \le n: |k-k_n| \le u b(k_n)} k^\lambda | g_k(n) - \mu^{-1} h_n(k) |
 \to 0.
\end{equation}

(ii) Assume that $k \le n$ is such that $|k-k_n| > u b(k_n)$ for some $u \ge 2$ and $n \ge n_0$,
where $n_0$ has been chosen large enough so that $k_n \ge \frac12 \mu^{-1} n$ for all $n \ge n_0$
(this is possible by virtue of \reff{eq:SojournMean} and Lemma~\ref{lem:InverseMuLimit}). Observe
first that, because $\omega_k^{-1} \ge 1$ and $|k-k_n| \ge 2$, we see by
Lemma~\ref{lem:InverseMuBound} that $|\mu_k - n| \ge |k - k_n|-1 \ge \frac12 |k-k_n|$. Hence, the
exponent in the definition of $g_k(n)$ is bounded by
\[
 \frac{(n-\mu_k)^2}{2 n \sigma^2 / \mu}
 \ge c_1 \frac{(k-k_n)^2}{k_n}
 \ge c_1 u^2 \log k_n
 \ge - c_1 u^2 \log (2 \mu) + c_1 u^2 \log n,
\]
where $c_1 = (16 \sigma^2)^{-1}$. The same bound is also valid for the exponent in the definition
of $h_n(k)$, because $\mu^3 > \mu$. Therefore, we obtain
\begin{equation}
 \label{eq:fh3}
 n^{1/2} k^\lambda | g_k(n) - \mu^{-1} h_n(k)|
 \le c_2 (2\mu)^{c_1 u^2} n^{\lambda - c_1 u^2},
\end{equation}
where $c_2 = 2 (2\pi \sigma^2 / \mu)^{-1/2}$. The right side above tends to zero as $n \to \infty$,
if in addition to $u \ge 2$, we also require that $u > (\lambda/c_1)^{1/2}$. The proof is now
completed by combining \reff{eq:fh2} and~\reff{eq:fh3}.
\end{proof}

\begin{lem}
\label{lem:ExponentialBound}
For any $\lambda \ge 0$ and any environment $\omega$ satisfying \reff{eq:SojournMean}, there
exists a constant $c>0$ such that the functions $h_n(k)$ defined by~\reff{eq:h} satisfy
\[
 \sup_{k>n} k^\lambda h_n(k) = o(e^{-cn}).
\]
\end{lem}
\begin{proof}
Recalling (by \eqref{eq:SojournMean} and Lemma~\ref{lem:InverseMuLimit}) that $k_n/n \to \mu^{-1}
< 1$, we may fix a positive constant $c_1 < 1$ and an integer $n_0$ such that $k_n \le c_1 n$ for
all $n \ge n_0$. Assume now that $k > n$ and $n \ge n_0$. Then, $k-k_n \ge (1-c_1)k$, so that the
exponent in the expression of $h_n(k)$ is bounded by
\[
 \frac{(k-k_n)^2}{2 n \sigma^2/\mu^3}
 \ge \frac{(1-c_1)^2}{2 \sigma^2/\mu^3} k.
\]
As a consequence,
\[
 k^\lambda h_n(k) \le c_2 n^{-1/2} k^\lambda e^{- c_3 k},
\]
where $c_2 = (2\pi \sigma^2/\mu^3)^{-1/2}$ and $c_3 = (1-c_1)^2 (2 \sigma^2/\mu^3)^{-1}$. Because
the function $t \mapsto t^\lambda e^{-c_3 t}$ is decreasing on the interval $[\lambda/c_3,\infty)$,
we conclude that
\[
 \sup_{k > n} k^\lambda h_n(k) \le c_2 n^{\lambda-1/2} e^{- c_3 n},
\]
for all $n \ge \max(n_0,\lambda/c_3)$, so the claim follows.
\end{proof}

\subsection{Proof of Theorem~\ref{thm:LLT}}
\label{sec:ProofLLT}

By combining Lemmas \ref{lem:LLTHittingTimes2}, \ref{lem:fg}, and \ref{lem:gh} we see that
\[
 \max_{1 \le k \le n} k^\lambda \left| \prq(T_k = n) - \mu^{-1} h_n(k) \right| = o(n^{-1/2}),
\]
where the functions $h_n$ are defined by~\reff{eq:h}. Hence by Lemma~\ref{lem:ExponentialBound} and
the fact that $\prq(T_k = n) = 0$ for all $k > n$, we conclude that
\[
 \sup_{k \ge 1} k^\lambda \left| \prq(T_k = n) - \mu^{-1} h_n(k) \right| = o(n^{-1/2}).
\]
Further, because $\omega_k^{-1} \le c_1 k^\lambda \le c_1 (k+1)^\lambda$ by
\reff{eq:SojournGrowth}, we see by Lemma~\ref{lem:Space2Time} that
\[
 \sup_{k \ge 0} \left| \prq(X_n=k) -  (\omega_k^{-1}/\mu) h_{n+1}(k+1)) \right|
 = o(n^{-1/2}).
\]

We complete the proof of Theorem~\ref{thm:LLT} by showing below that
\begin{equation}
 \label{eq:hDiff}
 \sup_{k \ge 0} k^\lambda | h_{n+1}(k+1) - h_n(k) | = o(n^{-1/2}).
\end{equation}
Let us write $h_n(k) = c_1 n^{-1/2} e^{-c_2 \alpha_{k,n}^2}$, where $c_1 = (2\pi
\sigma^2/\mu^3)^{-1/2}$, $c_2 = (2 \sigma^2/\mu^3)^{-1}$, and $\alpha_{k,n} = n^{-1/2} (k-k_n)$.
Note that
\[
 \alpha_{k+1,n+1} - \alpha_{k,n}
 = (n+1)^{-1/2} (k_n -k_{n+1} + 1) + (k-k_n)( (n+1)^{-1/2} - n^{-1/2}).
\]
Note that $0 \le k_{n+1}-k_n \le 1$ (Lemma~\ref{lem:InverseMuBound}) and $|k-k_n| \le (1+c_3)n$
for all $k \le n$, where $c_3 = \sup_{\ell \ge 0}(k_\ell/\ell)$ is finite by
Lemma~\ref{lem:InverseMuLimit}. Therefore by applying the inequality $n^{-1/2} - (n+1)^{-1/2} \le
\frac12 n^{-3/2}$, we see that
\[
 |\alpha_{k+1,n+1} - \alpha_{k,n}| \le (2+(1+c_3)/2) n^{-1/2}
\]
for all $k \le n$. This estimate combined with the inequality $|e^{-x^2} - e^{-y^2}| \le |x-y|$
now shows that
\[
 \max_{k \le n} k^\lambda | h_{n+1}(k+1) - h_n(k) | = O(n^{\lambda-1}).
\]
Together with Lemma~\ref{lem:ExponentialBound}, we now conclude the validity of~\reff{eq:hDiff},
and the proof of Theorem~\ref{thm:LLT} is complete.
\qed

\subsection{Proof of Theorem~\ref{thm:CLT}}
\label{sec:ProofCLT}

It suffices to show that
\[
 \prq \! \left( x < \frac{X_n - k_n}{\tilde\sigma \sqrt n} \le y \right)
 \to \frac{1}{\sqrt{2\pi}} \int_x^y e^{-t^2/2} \, dt
\]
for all $x<y$, where $\tilde\sigma^2 = \sigma^2 / \mu^{3}$. Note that the left side above can be
written as $\sum_{k\in I_n} \prq(X_n=k)$, where the set $I_n = I_n(x,y)$ is defined by
\[
 I_n
 = \left\{k \in \Z: \floor{x \tilde\sigma \sqrt{n}} + 1 \le k - k_n \le \floor{y \tilde\sigma \sqrt{n}} \right\},
\]
and by Theorem~\ref{thm:LLT},
\[
 \sum_{k\in I_n} \prq(X_n=k) - \mu^{-1} \sum_{k \in I_n} \omega_k^{-1} h_n(k) \to 0.
\]

Note that for any real numbers $x < y$ and any sequence $(a_k)_{k \in \Z}$,
\begin{equation}
 \label{eq:DiscreteIntegral}
 \left| \int_x^y a_{\floor{t}} \, dt - \sum_{k=\floor{x}+1}^{\floor{y}} a_k \right|
 \le |a_{\floor{x}}| + |a_{\floor{y}}|.
\end{equation}
By applying~\reff{eq:DiscreteIntegral}, using \reff{eq:SojournGrowth}, and performing a change of
variables, we see that
\begin{align*}
 \sum_{k\in I_n} \omega_k^{-1} h_n(k)
 & = \int_{x \tilde\sigma \sqrt n}^{y \tilde\sigma \sqrt n} \omega_{k_n + \floor{u}}^{-1} h_n(k_n + \floor{u}) \,du  + O(n^{\lambda-1/2}) \\
 & = \tilde\sigma\sqrt n \int_x^y \omega^{-1}_{k_n + \floor{\tilde\sigma t \sqrt{n}}} h_n(k_n + \floor{\tilde\sigma t \sqrt{n}}) \, dt  + o(1),
\end{align*}
where $O(n^{\lambda-1/2}) = o(1)$ due to $\lambda < 1/2$. Further,   $| t^2 - \floor{t}^2 | \le 2
|t|$ shows that
\begin{align*}
 h_n(k_n + \floor{\tilde\sigma t \sqrt{n}})
 &= (2\pi \tilde\sigma^2 n)^{-1/2} \exp\left( - \frac{\floor{t \tilde\sigma \sqrt n}^2}{2n\tilde\sigma^2} \right) \\
 &= (2\pi \tilde\sigma^2 n)^{-1/2} e^{-t^2/2} O\bigl(e^{|t|/(\tilde\sigma\sqrt{n})}\bigr) \\
 &= (2\pi \tilde\sigma^2 n)^{-1/2} e^{-t^2/2}(1 + O(n^{-1/2})),
\end{align*}
because we assume that $t \in (x,y)$.

The next lemma finishes the proof by showing that
\[
 \frac{1}{\sqrt{2\pi}} \int_x^y (\omega_{k_n + \floor{\tilde\sigma t \sqrt{n}}}^{-1} / \mu) \, e^{-t^2/2} \, dt
 \to
 \frac{1}{\sqrt{2\pi}} \int_x^y e^{-t^2/2} \, dt.
\]
The intuition is that, restricting integration to sufficiently small subintervals of $(x,y)$ and making $n$ large, $e^{-t^2/2}$ is virtually constant, while the fluctuations of $\omega_{k_n + \floor{\tilde\sigma t \sqrt{n}}}^{-1}$ average out in the integral.
\begin{lem}\label{lem:weak_conv}
On the interval $(x,y)$, the probability measures defined by $\mathrm{m}_n(dt) = Z_n^{-1} \mu^{-1}
\omega_{k_n + \floor{\tilde\sigma t \sqrt{n}}}^{-1} \, dt$ converge weakly to the uniform probability
measure $\mathrm{m}(dt)=(y-x)^{-1}\,dt$ and the normalizing factor $Z_n \to y-x$.
\end{lem}
\begin{proof}
We must show that $\mathrm{m}_n((x,s])\to \mathrm{m}((x,s])$ for all $s\in(x,y)$. By a change of
variables and \reff{eq:DiscreteIntegral} we see that
\beqn
 \int_x^s \omega_{k_n + \floor{\tilde\sigma t \sqrt{n}}}^{-1} \, dt
 = \tilde\sigma^{-1} n^{-1/2} \sum_{k\in I_n(x,s)} \omega_k^{-1} + o(1),
\eeqn
because $\omega_k^{-1} = O(k^\lambda)$ and $\lambda < \frac12$. The center of $I_n(x,s)$ is $c_n=
\mu^{-1}n(1+o(1))$. For $u>0$ large enough and $b(n)$ as in~\eqref{eq:EnvMovingAverages},
$|I_n(x,s)| = \floor{s\tilde\sigma \sqrt n}-\floor{x \tilde\sigma \sqrt n} < u b(n)$. Therefore,
\eqref{eq:EnvMovingAverages} implies that
\[
 n^{-1/2} \!\!\!\! \sum_{k\in I_n(x,s)} \omega_k^{-1}
 = n^{-1/2} |I_n(x,s)| \mu + o(1).
\]
Because $n^{-1/2} |I_n(x,s)| \to (s-x) \tilde\sigma$, we get $\mu^{-1} \int_x^s \omega_{k_n + \floor{\tilde\sigma
t \sqrt{n}}}^{-1} \, dt \to s-x$, and $Z_n \to y-x$ follows by taking $s=y$. This finishes the proof.
\end{proof}

\section{Proofs for quenched random environments}
\label{sec:ProofsRandom}

In this section we prove Theorem~\ref{thm:Quenched} for the random walk $(X_n)$ in a quenched
random environment satisfying the assumptions \reff{eq:EnvMoments}--\reff{eq:EnvMixing}.
Section~\ref{sec:StationarySequences} contains some preliminary facts on the growth rate and
moving averages of stationary sequences, and the proof of Theorem~\ref{thm:Quenched} is given in
Section~\ref{sec:ProofOfQuenchedLLT}.

\subsection{Growth rate and moving averages of stationary sequences}
\label{sec:StationarySequences}

The following result establishes a bound on the growth rate of a stationary sequence in terms of
its moments.

\begin{lem}
\label{lem:GrowthRate}
Let $(\xi_0,\xi_1,\dots)$ be a stationary random sequence such that $\Ee |\xi_0|^q < \infty$ for
some $q > 0$. Then $\xi_k = o(k^{1/q})$ almost surely.
\end{lem}
\begin{proof}
Fix $\epsilon>0$. By stationarity and Fubini's theorem we see that
\[
 \sum_{k=0}^\infty \pre( k^{-1/q} |\xi_k| > \epsilon)
 = \sum_{k=0}^\infty \pre( (\epsilon^{-1} |\xi_0|)^q > k)
 = \Ee \floor{(\epsilon^{-1} |\xi_0|)^q}
 < \infty,
\]
where $\floor{x}$ denotes the integer part of $x$. Because $\epsilon$ was arbitrarily chosen, the
claim follows by the  Borel--Cantelli lemma.
\end{proof}

The next result analyzes the moving averages of a stationary random sequence in terms of its
moments and mixing rate. The result, proven with the help Peligrad's law of large numbers
\cite{Peligrad1985} (see also Bingham~\cite{Bingham1986a} for a nice survey), is similar in spirit
to Kiesel~\cite[Thm 1]{Kiesel1998}, but tailored to fit our needs.

\begin{lem}
\label{lem:fast}
Let $(\xi_0,\xi_1,\dots)$ be a stationary random sequence such that $\Ee |\xi_0|^q < \infty$ for
some $q \ge 1$, and for which the mixing coefficients as defined in~\reff{eq:MixingCoefficients}
satisfy
\begin{equation}
 \label{eq:PeligradMixing}
 \sum_{n=1}^\infty \phi^{1/\kappa}(2^n) < \infty
 \quad \text{for some $\kappa \ge 2$}.
\end{equation}
Then
\begin{equation}
 \label{eq:fast}
 \max_{1 \le j \le u k^s} \left| \sum_{\ell=k+1}^{k+j} (\xi_\ell - \Ee \xi_0) \right| = o(k^r)
\end{equation}
almost surely for any $s>0$, $u>0$, and
\begin{equation}
 \label{eq:AssumptionOnR}
 r \ge \max \left( \frac{1+s}{q}, \ \frac{s}{2} + \frac{1}{\kappa-1} \right).
\end{equation}
\end{lem}
\begin{proof}
Denote the left side of~\reff{eq:fast} by $M_k$, and fix an arbitrary $\epsilon > 0$. Note that by
stationarity,
\begin{equation}
 \label{eq:fast1}
 \sum_{k=1}^\infty \pre(k^{-r} M_k > \epsilon)
 = \sum_{k=1}^\infty \pre\Bigl( \max_{1 \le j \le u k^s} |S_j| > \epsilon k^r\Bigr),
\end{equation}
where $S_j = \sum_{\ell=1}^j (\xi_\ell- \Ee \xi_0)$. The right side in~\reff{eq:fast1} is finite
if and only if
\[
 \int_1^\infty \pre\Bigl( \max_{1 \le j \le u t^{s}} |S_j| > \epsilon t^{r}\Bigr) \, dt < \infty,
\]
and by a change of variables, the above is further equivalent to
\[
 \int_1^\infty t^{1/s-1} \pre\Bigl( \max_{1 \le j \le t} |S_j| > \epsilon_1 t^{r/s}\Bigr) \, dt < \infty,
\]
where $\epsilon_1 = u^{-r/s} \epsilon$. We conclude that $\sum_{k=1}^\infty \pre(k^{-r} M_k >
\epsilon)$ is finite if and only if
\begin{equation}
 \label{eq:PeligradSum}
 \sum_{n=1}^\infty n^{\alpha p - 2} \pre\Bigl( \max_{1 \le j \le n} |S_j| > \epsilon_1 n^{\alpha}\Bigr) < \infty,
\end{equation}
where $\alpha = r/s$ and $p = (1+s)/r$.

The parameters $r$ and $s$ have been transformed into $\alpha$ and $p$ to conform with the
notations used in~\cite{Peligrad1985}. Observe that $\alpha p > 1$, and moreover,
\reff{eq:AssumptionOnR} implies that $p \le q$, $\alpha
> 1/2$, and $[(\alpha p-1)/(\alpha-1/2)] + 1 \le \kappa$. Hence, Peligrad's law of large
numbers \cite[Thm 2]{Peligrad1985} yields the validity of~\reff{eq:PeligradSum}, and therefore,
$\sum_{k=1}^\infty \pre(k^{-r} M_k > \epsilon)$ is finite. Because $\epsilon$ was arbitrarily
chosen, the claim follows by the Borel--Cantelli lemma.
\end{proof}

\subsection{Proof of Theorem~\ref{thm:Quenched}}
\label{sec:ProofOfQuenchedLLT}

By virtue of Theorems~\ref{thm:LLN}--\ref{thm:CLT}, we only need to verify that the regularity
conditions \reff{eq:SojournGrowth} -- \reff{eq:EnvMovingAverages} are valid for $\pre$-almost
every realization of the random environment.

Recall that $\Ee \omega_0^{-q} < \infty$ for some $q>5$ by~\reff{eq:EnvMoments}. Consequently,
Lemma~\ref{lem:GrowthRate} implies that $\omega_k^{-1} = o(k^{1/q})$, so that especially,
\reff{eq:SojournGrowth} holds with $\lambda = 1/q$. Because $q>5$, the random variables $m_k =
\omega_k^{-1}$ and $s^2_k = (1-\omega_k)\omega_k^{-2}$ have finite second moments, so the
assumption~\reff{eq:EnvMixing} on the mixing coefficients implies that the stationary random
sequences $(m_k)$ and $(s^2_k)$ satisfy the law of the iterated logarithm \cite[Section
12]{Bingham1986}. As a consequence, \reff{eq:SojournMean} and \reff{eq:SojournVariance} are valid
$\pre$-almost surely with $\mu = \Ee m_0 > 1$ and $\sigma^2 = \Ee s_0^2 > 0$.

Because $\Ee \omega_0^{-3}$ is finite, the pointwise ergodic theorem implies that \linebreak $k^{-1}
\sum_{j=0}^{k-1} \omega_k^{-3}$ converges, so that \reff{eq:SojournThirdMoments} holds.

To verify \reff{eq:EnvMovingAverages}, note that because $\phi$ is decreasing, the sum in
\reff{eq:PeligradMixing} equals $\int_1^\infty \phi^{1/\kappa}(2^{\floor{x}}) \, dx \le
\int_1^\infty \phi^{1/\kappa}(\floor{2^{x-1}}) \, dx = (\log 2)^{-1}\int_1^\infty
\phi^{1/\kappa}(\floor{y}) y^{-1} \, dy$. Using H\"older's inequality, the last integral can be
bounded from above by $\left(\int_1^\infty \phi^{1/2}(\floor{y}) \,dy \right)^{2/\kappa}
\left(\int_1^\infty y^{-p}\,dy\right)^{1/p}$ with $p=(1-2/\kappa)^{-1}$, which is finite due to
\reff{eq:EnvMixing}. Therefore, condition \reff{eq:PeligradMixing} of Lemma~\ref{lem:fast} holds
for any $\kappa > 2$. Next, because $q>5$, we may choose an exponent $s > 1/2$ such that $s < 1 -
2q^{-1}$ and $s \le q/2 - 2$. Our choice of $s$ implies that the exponent $r = 1/2 - q^{-1}$
satisfies~\reff{eq:AssumptionOnR} for some large enough $\kappa$. Hence by Lemma~\ref{lem:fast}, it
follows that
\[
 \max_{j:|j| \le u k^s}
 \left| \sum_{\ell=k}^{k+j-1}(\omega_\ell^{-1}-\mu) \right| = o(k^r)
\]
almost surely for all $u>0$. Because $(k\log k)^{1/2} = o(k^s)$, it follows that
\reff{eq:EnvMovingAverages} holds with $\lambda = 1/q$.
\qed

\section{Conclusions}
\label{sec:Conclusions}

We studied the propagation of a particle in a one-dimensional inhomogeneous medium, where the
motion is induced by chaotic and fully deterministic local rules, and the initial condition is the
sole source of randomness. The spatially varying local rules constitute an environment, which is
frozen during the particle's lifetime. This model falls into the framework of extended dynamical
systems which lack physically observable invariant measures. Defining the local rules via
piecewise affine maps allows to reduce the model to a simple unidirectional random walk on the
integers.

The main result of the paper shows that the probability mass function of the random walk
approaches a modulated Gaussian density, where the modulating factor is explicitly given in terms
of the local properties of the environment. In contrast, when looking at the walk over a coarser
diffusive space scale, the non-Gaussian modulating factor averages out asymptotically, and the
distribution of the walk approaches a standard Gaussian distribution.

Although our analysis is restricted to a special instance of a random walk, we believe that the
obtained results could serve as useful benchmarks when testing hypotheses concerning more general
extended dynamical systems and random walks in random environments.

\appendix

\section{A generalized inverse}
\label{sec:GeneralizedInverse}

Let $(a(k))_{k=0}^\infty$ be an increasing sequence such that $a(0) = 0$ and \linebreak $\lim_{k\to\infty}
a(k) = \infty$, and define its generalized inverse by
\begin{equation}
  \label{eq:InverseMu}
  \ainv(n) = \min\{k \in \Z_+: a(k) \ge n\}.
\end{equation}
Then also $\ainv(0) = 0$ and $\lim_{n\to\infty} \ainv(n) = \infty$. The following result
summarizes some basic properties of the inverse.

\begin{lem}
\label{lem:InverseMuLimit}
If $\lim_{k \to \infty} a(k)/k = \mu$, then $\lim_{n \to \infty} \ainv(n)/n = 1/\mu$.
\end{lem}
\begin{proof}
By definition, $a(\ainv(n)-1) < n \le a(\ainv(n))$ for all $n$ such that $\ainv(n) > 0$, so that
\[
 \frac{a(\ainv(n)-1)}{\ainv(n)} < \frac{n}{\ainv(n)} \le \frac{a(\ainv(n))}{\ainv(n)}.
\]
Because $\ainv(n) \to \infty$, it follows that $a(\ainv(n))/\ainv(n) \to \mu$. Therefore, the
above bounds imply $n/\ainv(n) \to \mu$, and the proof is complete.
\end{proof}

\begin{lem}
\label{lem:InverseMuBound}
Assume that $c = \inf_k (a(k+1) - a(k)) > 0$. Then for all positive integers $k$ and $n$,
\begin{equation}
 \label{eq:InverseMuBound1}
 |a(k) - n| \ge c(|k-\ainv(n)|-1),
\end{equation}
and
\begin{equation}
 \label{eq:InverseMuBound2}
 \ainv(n+1) - \ainv(n) \le \ceil{ c^{-1}}.
\end{equation}
\end{lem}
\begin{proof}
If $k \ge \ainv(n)$, then $a(k) - n \ge a(k) - a(\ainv(n)) \ge c(k-\ainv(n))$. If $k < \ainv(n)$,
then $n - a(k) > a(\ainv(n)-1) - a(k) \ge c(\ainv(n) - k - 1)$. Hence
\reff{eq:InverseMuBound1} follows.

To prove \reff{eq:InverseMuBound2}, fix a positive integer $n$, and denote $\ell = \ainv(n)$. Then
$a(\ell) \ge n$ and $a(\ell + \ceil{c^{-1}}) - a(\ell) \ge c \ceil{c^{-1}} \ge 1$, so that $a(\ell
+ \ceil{c^{-1}}) \ge n+1$. Hence, $\ainv(n+1) \le \ell + \ceil{c^{-1}}$.
\end{proof}

\section*{Acknowledgements}

We thank Timo Sepp\"al\"ainen, Esko Valkeila, Bastien Fernandez, and Milton Jara for fruitful
discussions. We are indebted to Tapio Simula, whose simulations motivated our work. We also thank the
anonymous referees for sharp remarks and helpful comments on improving the presentation of this
paper. Both authors have been supported by a fellowship from the Academy of Finland.

\bibliographystyle{abbrv}
\bibliography{RelatedWork}

\end{document}